\documentclass[12pt]{amsproc}

\author{Jhone Caldeira}
\address{Instituto de Matem\'atica e Estat\'istica, Universidade Federal de Goi\'as,
Goi\^ania-GO, 74690-900, Brazil}
\email{jhone@ufg.br}

\author{Emerson de Melo}
\address{Department of Mathematics, University of Bras\'ilia, Bras\'ilia-DF 70910-900, Brazil}
\email{emerson@mat.unb.br}

\thanks{This paper is dedicated to the 60th birthday of Professor Pavel Shumyatsky. 
We are grateful for his inspiring ideas on this research topic.}

\title{Non-metacyclic groups acting with nilpotent centralizers}
\date{2021}

\newtheorem{theorem}{\sc Theorem}[section]
\newtheorem{lemma}[theorem]{\sc Lemma}

\newtheorem{remark}[theorem]{\sc Remark}

\begin{document}


\begin{abstract}
Let $A$ be a non-metacyclic finite group. Suppose that $A$ acts coprimely on a finite group $G$ 
in such a manner that $C_G(a)$ is nilpotent for any $a\in A^{\#}$. 
In the present paper we investigate some conditions on $A$ which imply that $G$ is 
nilpotent with ``bounded'' nilpotency class. 
More precisely, we generalize known results on action of $q$-groups and Frobenius groups, where $q$ is a prime.

\end{abstract}

\maketitle

\section{Introduction}

Suppose that a finite group $A$ acts by automorphisms on a finite group $G$.  The action is coprime if the groups $A$ and $G$ have coprime orders. We denote by $C_G(a)$ the set $\{g\in G\ |\ g^a=g \},$ the centralizer of $a$ in $G$ (the fixed-point subgroup). Similarly, we denote by $C_G(A)$ the set $\{g\in G\ |\ g^a=g \ \textrm{for all} \ a\in A\}.$ In what follows we denote by $A^\#$ the set of nontrivial elements of $A$. If $A$ is metacyclic and $C_G(a)$ is nilpotent for any $a\in A^{\#}$, it follows from the classification of finite simple groups that the group $G$ is soluble (see \cite{WC}). Furthermore, we can easily construct examples showing that $G$ need not be nilpotent even if the action is coprime. On the other hand, when $A$ is not metacyclic and acts coprimely on a finite group $G$, recent results have shown that the properties of $G$ are in a certain sense close to the corresponding properties of the centralizers of non trivial elements. 

Recently, prompted by Mazurov's problem 17.72 in the Kourovka 
Notebook \cite{KourNot}, some attention was given to the situation where a Frobenius group $FH$ acts by 
automorphisms on a finite group $G$. Recall that a Frobenius group $FH$ with kernel 
$F$ and complement $H$ can be characterized 
as a finite group that is a semidirect product of a normal subgroup $F$ by $H$ such that 
$C_F(h) =1$ for every $ h \in H \setminus \{1\}$. By Thompson's theorem \cite{T} the kernel $F$ 
is nilpotent, and by Higman's theorem \cite{Hig} the nilpotency class of $F$ 
is bounded in terms of the least prime divisor of $|H|$.

In the following we present some results about coprime action of non-metacyclic groups with nilpotent centralizers.

1. Let $q$ be a prime. If $A$ is an elementary abelian $q$-group of rank 3 and $C_G(a)$ is nilpotent 
of class $c$ for any $a\in A^\#$, then the group $G$ is nilpotent with class bounded solely in terms of $c$ and $q$ \cite{War,Sh}.

2. In the recent article \cite{Eme1} the above result was extended to the case where $A$ is not necessarily abelian. Namely, it was shown that if $A$ is a finite group of prime exponent $q$ and order at least $q^3$ acting on a finite $q'$-group $G$ in such a manner that $C_G(a)$ is nilpotent of class at most $c$ for any $a\in A^{\#}$, then $G$ is nilpotent with class bounded solely in terms of $c$ and $q$. In \cite{Eme2}, this result has been extended in the following way: if $\gamma_{\infty} (C_{G}(a))$ has order at most $m$ for any $a \in A^{\#}$, then the order of $\gamma_{\infty} (G)$ is bounded solely in terms of $m$. Furthermore, 
it was shown that if the Fitting subgroup of $C_{G}(a)$ has index at most $m$ for any $a \in A^{\#}$, then the second Fitting subgroup of $G$ has index bounded solely in terms of $m$.

3. Let $p$ be a prime and $A=FH$ a Frobenius group with complement $H$ and whose kernel $F$ is an elementary abelian group of order $p^2$. If $FH$ acts coprimely on a finite group $G$ in such a manner that  $C_G(H)$ and $C_G(x)$ are nilpotent for any $x\in F^{\#}$, then $G$ is nilpotent. In addition, there exists a constant $P = P(c, |H|)$ depending only on $c$ and $|H|$ with the following property: if $p\geq P$, $C_G(H)$ and $C_G(x)$ are nilpotent of class at most $c$ for any $x\in F^{\#}$, then $G$ is nilpotent with class bounded solely in terms of $c$ and $|H|$ \cite{AS}. Examples in \cite{KhuMakShu} show that the result on the nilpotency class of $G$ is no longer true if $p=2$. 

4. Let $p$ be a prime and $A=FH$ a Frobenius group with complement $H$ and whose kernel $F$ is an elementary abelian group of order $p^3$. If $FH$ acts coprimely on a finite
group $G$ in such a manner that  $C_G(H)$ and $C_G(x)$ are nilpotent of class at most $c$ for any $x\in F^{\#}$, then $G$ is nilpotent with class bounded solely in terms of $c$ and $|H|$ \cite{CalMelShu}.

The next result generalizes the results mentioned in Item (2). It is an immediate consequence of the theory of regular $q$-groups and well-known results on metacyclic $q$-groups. In fact, in Lemma \ref{C1} we prove that if $q>3$ is a prime and $A$ is a non-metacyclic $q$-group, then $A$ contains a subgroup of exponent $q$ and order $q^3$ which is sufficient to obtain the next theorem since the proofs in \cite{Eme1} and \cite{Eme2} are reduced to the case where $A$ is a group of exponent $q$ and order $q^3$. Throughout the paper we use the expression ``$(a,b,\dots )$-bounded'' to abbreviate ``bounded from above in terms of  $a,b,\dots$ only''.





\begin{theorem}\label{main1}
Let $q>3$ be a prime and $A$ a non-metacyclic finite $q$-group acting by automorphisms on a finite $q'$-group $G$. 
\begin{enumerate}
\item[i)] If $C_{G}(a)$ is nilpotent of class at most $c$ for any $a\in A^{\#}$, then $G$ is nilpotent with $(c,q)$-bounded class. 
\item[ii)] If $\gamma_{\infty} (C_{G}(a))$ has order at most $m$ for any $a \in A^{\#}$, then the order of $\gamma_{\infty} (G)$ is bounded solely in terms of $m$. 
\item[iii)] If the Fitting subgroup of $C_{G}(a)$ has index at most $m$ for any $a \in A^{\#}$, then the second Fitting subgroup of $G$ has index bounded solely in terms of $m$.
\end{enumerate}   
\end{theorem}

As a consequence of the classification of groups with no noncyclic characteristic abelian subgroups we prove in Lemma \ref{submet} that if $A=FH$ is a non-metacyclic Frobenius group, then the kernel $F$ contains a normal elementary abelian subgroup of rank 2. We use Theorem \ref{main1} and this fact to extend the result mentioned in Item (4) to the case where $F$ is not necessarily abelian.


\begin{theorem}\label{main3}
Let $FH$ be a Frobenius group with a non-metacyclic kernel $F$ of odd order and complement $H$. Suppose that $FH$ acts coprimely on a finite
group $G$ in such a manner that $C_G(H)$ and $C_G(x)$ are nilpotent of class at most $c$ for any $x\in F^{\#}$. Then $G$ is nilpotent 
of $(c,|H|)$-bounded class.
\end{theorem}

Although we do not know  if the nilpotency class would not be bounded when the kernel $F$ is a non-metacyclic group of even order, it is possible to construct a Frobenius group with a non-metacyclic kernel of order 16 which does not contain a subgroup of exponent 2 and order 8. The next result generalizes Item (3) when $FH$ is a supersolvable group. Recall that a group $N$ is supersolvable if $N$ possesses a normal series with cyclic factors such that each term is normal in $N$. It is easy to see that a supersolvable group possesses a 
chief series whose factors have prime order.

\begin{theorem}\label{main2}
Let $FH$ be a supersolvable non-metacyclic Frobenius group with kernel $F$ and complement $H$. Suppose that $FH$ acts coprimely on a finite
group $G$ in such a manner that  $C_G(H)$ and $C_G(x)$ are nilpotent of class at most $c$ for any $x\in F^{\#}$. Then $G$ is nilpotent 
of $(c,|H|)$-bounded class.
\end{theorem}

The paper is organized as follows. In Section 2, we summarize without
proofs some basic results on coprime actions and we prove Theorem \ref{main1}  and Theorem \ref{main3}. In the third section, we prove some results on Lie rings admitting a Frobenius group of automorphisms. The proof of Theorem \ref{main2} is given in Section 4.

\section{Preliminaries}

If $A$ is a group of automorphisms of a group $G$, the subgroup generated by elements of the form $g^{-1}g^\alpha$ with $g\in G$ and $\alpha\in A$ is denoted by $[G,A]$. It is well-known that the subgroup $[G,A]$ is an $A$-invariant normal subgroup of $G$. Our first lemma is a collection of well-known facts on coprime actions (see for example \cite{GO}). Throughout the paper we will use it without explicit references.
\begin{lemma}
Let  $A$ be a group of automorphisms of a finite group $G$ such that $(| G |, | A |) = 1$. Then:
\begin{enumerate}
\item[i)] $G = C_{G}(A)[G,A]$.
\item[ii)] $[G,A,A]=[G,A]$.
\item[iii)] $A$ leaves invariant some Sylow $p$-subgroup of $G$ for each prime $p\in\pi(G)$.
\item[iv)] $C_{G/N}(A)=C_G(A)N/N$ for any $A$-invariant normal subgroup $N$ of $G$.
\item[v)] If $G$ is nilpotent and $A$ is a noncyclic abelian group, then $G=\prod_{a\in A^{\#}}C_{G}(a)$.
\end{enumerate}
\end{lemma}

We will require the following well-known characterization of metacyclic groups (see for example \cite[Theorem 2.6]{Bla}).

\begin{lemma}\label{c}
For $q$ an odd prime, a $q$-group $A$ is metacyclic if and only if $[A:A^q]\leq q^2$.
\end{lemma}

We will also use the following fact about regular  $q$-groups: if $G$ is a $q$-group for some prime $q$ and $|G|\leq q^q$, then $G$ is a regular $q$-group. In particular, $|G/G^q|=|\Omega_1(G)|$.

\begin{lemma}\label{C1} Let $q>3$ be a prime and $A$ a non-metacyclic group. Then $A$ contains a subgroup of exponent $q$ and order $q^3$.
\end{lemma}
\begin{proof}
We use induction on the order of $A^q$. If $A^q=1$ there is nothing to prove since $A$ is non-metacyclic. Otherwise we can find a normal subgroup $N$ of $A$ of order $q$, which is contained in $A^q$. It follows from the inductive hypothesis that $A/N$ has a subgroup $K$ such that $N\leq K$ and $K/N$ has exponent $q$ and order $q^3$. Since $N$ has order $q$ we obtain that $K$ has order $q^4$. Thus $K$ is a regular group and $\Omega_1(K)\geq q^3$ since $K/K^q$ has order at least $q^3$.  
\end{proof}

The next theorem is proved in \cite[Theorem 5.5.3]{GO}.
\begin{theorem}\label{noncyclic}
Let $q$ be an odd prime. Let $Q$ be a $q$-group which contains no noncyclic characteristic abelian subgroups. Then $Q$ is isomorphic to the central product of a cyclic group and an extra-special group of exponent $q$.
\end{theorem}

Now we apply Theorem \ref{noncyclic} to non-metacyclic Frobenius groups.

\begin{lemma}\label{submet}
Let $FH$ be a non-metacyclic Frobenius group with complement $H$ and kernel $F$. Then $F$ contains a normal elementary abelian $p$-subgroup  of order $p^2$ for some prime $p$.
\end{lemma}
\begin{proof}
Since $F$ is nilpotent we can assume that $F$ is a $p$-group for some prime $p$. If $p=2$, then $Z(F)$ must not be cyclic and the result follows. If $p$ is odd and contains a noncyclic characteristic  abelian subgroup the result also follows. Thus, it remains to prove the case when $F$ contains no noncyclic characteristic  abelian subgroup. In this case by Theorem \ref{noncyclic}, $F$ is a central product of a cyclic group and an extra-special group of exponent $p$. Therefore, $\Omega_1(F)$ has exponent $p$ and is not cyclic. The proof is complete.
\end{proof}

We are now in a position to show Theorem \ref{main1} and Theorem \ref{main3}. As mentioned in the introduction, Theorem \ref{main1} is an immediate consequence of Lemma \ref{C1} and the main results of \cite{Eme1,Eme2}.

Now we prove Theorem \ref{main3}. Let us assume the hypothesis of the theorem. Thus, $FH$ is a Frobenius group with a non-metacyclic kernel $F$ of odd order and complement $H$ acting coprimely on a finite
group $G$ in such a manner that $C_G(H)$ and $C_G(x)$ are nilpotent of class at most $c$ for any $x\in F^{\#}$. We wish to show that $G$ is nilpotent of $(c,|H|)$-bounded class. Since $F$ is nilpotent, it is sufficient to consider the case where $F$ is a $p$-group for some prime $p$. By Lemma \ref{noncyclic} $F$ contains a normal elementary abelian $p$-subgroup $N$ of order $p^2$. Thus, considering the Frobenius group $NH$ and applying \cite[Theorem 1.5]{AS} we obtain that there exists a constant $P = P(c, |H|)$ depending only on $c$ and $|H|$ such that if $p\geq P$,  then $G$ is nilpotent with class bounded solely in terms of $c$ and $|H|$. On the other hand, if $3<p<P$, then we can apply Theorem \ref{main1} on $F$ to obtain that $G$ is nilpotent with class bounded solely in terms of $c$ and $|H|$. Now, assume that $p=3$. We will prove that the rank of $Z(F)$ is at least 3.  If the rank of $Z(F)$ is less than 3, we consider a subgroup of $Z(F)$ isomorphic to $C_3$ or $C_3\times C_3$ and then since $H$ acts on such subgroup we must have that $H$ contains an element of order 2. Thus, in this case $F$ is an abelian group of rank at most 2 which means that $F$ is metacyclic that is a contradiction. Therefore, $Z(F)$ contains an elementary abelian group of rank at most 3 and the result follows by \cite{Eme1} or \cite{War,Sh}.

\section{Bounding nilpotency class of Lie algebras}

We start this section with a criteria for nilpotency of graded Lie algebras. 
Let $p$ be a prime number and $L$ be a $\mathbb{Z}/p\mathbb{Z}$-graded 
Lie algebra. If $J, Y,J_1,\ldots, J_s$ are subsets of $L$ we use $[J,Y]$ to denote the subspace 
of $L$ spanned by the set $\{[j,y] \ ; \ j\in J,y\in Y\}$ and for $i \geq 2$ we write 
$[J_1,\ldots,J_i]$ for $[[J_1,\ldots,J_{i-1}],J_i]$.

Assume that there exist non-negative integers $u$ and $v$ such that 
\begin{equation}\label{eqq1}
[L,\underbrace{L_0,\ldots,L_0}_u]=0
\end{equation}
\noindent and
\begin{equation}\label{eqq2}
[[L,L] \cap L_0,\underbrace{L_a,\ldots,L_a}_{v}]=0, \ \mbox{for all} \ a \in \mathbb{Z}/p\mathbb{Z}.
\end{equation}

In \cite{CalMel1} we showed that conditions (\ref{eqq1}) and (\ref{eqq2}) together are 
sufficient to conclude that $L$ is nilpotent of $(p,u,v)$-bounded class. Here we repeat some arguments 
for the reader's convenience.

By \cite[Theorem 1]{Khu}, condition (\ref{eqq1}) implies that $L$ is solvable with $(p,u)$-bounded 
derived length. Thus, we can use induction on the derived 
length of $L$. If $L$ is abelian, there is nothing to prove. Assume that $L$ is metabelian. 
In this case, $[x, y, z] = [x, z, y]$, for every $x \in [L, L]$ and $y,z \in L$.

For each $b \in \mathbb{Z}/p\mathbb{Z}$ we denote $[L,L] \cap L_b$ by $L_b'$. 
By \cite[Proposition 2]{KhuShu} it is sufficient to prove that there exists a $(p,u,v)$-bounded number $t$ such that

$$[L_b',\underbrace{L_a,\ldots,L_a}_{t}]=0, \ \mbox{for all} \ a,b \in \mathbb{Z}/p\mathbb{Z}.$$

If $b=0$, this follows from (\ref{eqq2}) with $t=v$. Suppose that $b \neq 0$.
If $a=0$, the commutator is zero from (\ref{eqq1}) with $t=u$. In the case where $a \neq 0$ 
we can find a positive integer $s < p$ such that $b + sa = 0 \ (\textrm{mod }p)$. Therefore,
we have $[L_b',\underbrace{L_a,\ldots,L_a}_{v + s}] \subseteq [L_0',\underbrace{L_a,\ldots,L_a}_v]$. 
We know that the latter commutator is 0 by (\ref{eqq2}). 


In what follows $FH$ denotes a finite supersolvable Frobenius group with kernel $F$ 
and complement $H$. Moreover, we assume that $F$ is an elementary abelian $p$-group of rank 2 for some prime $p$.

Let $R$ be an associative ring with unity. Assume that the characteristic 
of $R$ is coprime with $p$ and the additive group of $R$ is finite (or locally finite). 
Let $L$ be a Lie algebra over $R$. 
The main goal of this section is to prove the following theorem.

\begin{theorem}\label{Liealgthe}
Suppose that $FH$ acts by automorphisms on $L$ in such a way that $C_{L}(H)$ is nilpotent of class $c$ and $C_L(x)$ is nilpotent of class $d$ for any $x \in F^{\#}$. Then $L$ is nilpotent of $(c,d,|FH|)$-bounded class.
\end{theorem}

Let $Z$ be a subgroup of prime order $p$ of $F$ such that $Z \triangleleft FH$ and let $\varphi$ be a generator of $Z$.

Now, let $\omega$ be a primitive $p$th root of unity. We extend the ground ring of $L$ by 
$\omega$ and denote by $\tilde{L}$ the algebra $L\otimes_{\mathbb{Z}}\mathbb{Z}[\omega]$. 
The action of $FH$ on $L$ extends naturally to $\tilde{L}$. Note that $C_L(x)$ is nilpotent of class $d$ for any $x \in F^{\#}$ and 
$C_{\tilde{L}}(H)$ is nilpotent of class $c$. Since $L$ and $\tilde{L}$ have the same nilpotency class, 
it is sufficient to bound the class of $\tilde{L}$. Hence, without loss of generality 
it will be assumed that the ground ring contains $\omega$ so that we will work with $L$ 
rather than $\tilde{L}$.

For each $i=0,\ldots,p-1$ we denote by $L_i$ the  $\varphi$-eigenspace corresponding 
to eigenvalue $\omega^i$, that is,  $L_i=\{x \in L \ ; \  x^{\varphi}=\omega^ix\}$. 
We have  $$L=\bigoplus_{i=0}^{p-1}L_i \ \textrm{ and } \ [L_i,L_j]\subseteq L_{i+j(\textrm{mod} \ p)}. $$

Note that  $L_0=C_L(Z)$. It is easy to see that $FH/Z$ acts on $L_0$ and $C_{L_0}(F/Z)=C_{L}(F)$. Also, note that $L_0$ is 
nilpotent of class $d$ by hypothesis.

A proof of the next lemma can be found in \cite[Lemma 2.4]{KhuMakShu}. 
It will be useful to decompose the subalgebra $L_0$ into $C_L(F)$ and $C_L(ZH^f)$, where $f\in F$.

\begin{lemma}\label{lemma dec Frb}
Suppose that a finite group $N$ admits a Frobenius group of automorphisms $KB$ with 
kernel $K$ and complement $B$ such that $C_N(K)=1$. Then $N=\langle C_N(B)^y ; \ y \in K \rangle$.
\end{lemma}

In the above lemma we can write $N=\langle C_N(B^y) ; \ y \in K \rangle$, 
since $C_N(B)^y=C_N(B^y)$.

For any $f \in F$ we denote by $V_f$ the subalgebra $C_L(ZH^f)\leq L_0$. 
Note that $ZH^f$ is also a Frobenius group and that $ZH^f=ZH$ whenever $f \in Z$. 

\begin{lemma}\label{lemma dec L0}
We have $L_0 = C_L(F)+\sum_{f\in F} V_f$.
\end{lemma}
\begin{proof}
The subalgebra $L_0$ is $FH$-invariant and $L_0=C_{L_0}(F)+[L_0,F]$. So it admits the natural action by the group $FH/Z$. 
Moreover, $F/Z$ acts fixed-point-freely on $[L_0,F]$. Since the additive group of $L$ is finite, 
it is immediate from Lemma \ref{lemma dec Frb} that 
$[L_0,F] =\langle C_{L_0}(H^{\tilde{f}}) ; \ \tilde{f} \in F/Z \rangle$. As a 
consequence we have that $[L_0,F] =\langle C_{L}(ZH^f) ; \ f \in F \rangle$ and then $L_0 = C_L(F)+\sum_{f\in F} V_f$.
\end{proof}

\begin{lemma}\label{lemma dec L}
We have $[L,\underbrace{C_L(F),\ldots,C_L(F)}_d]=0$.
\end{lemma}
\begin{proof}
Since $F$ is an elementary abelian $p$-group, 
it is immediate that $L=\sum_{x \in F^{\#}} C_L(x)$. 
Thus, taking into account that all centralizers $C_L(x)$ is nilpotent of class $d$ 
for any $x \in F^{\#}$ and $C_L(F) \leq C_L(x)$ whenever $x \in F^{\#}$, 
we conclude 
$$[L,\underbrace{C_L(F),\ldots,C_L(F)}_d] = \sum_{x\in F^{\#}}[C_L(x),\underbrace{C_L(F),\ldots,C_L(F)}_d] =0.$$
\end{proof}

It is clear that any conjugated $H^f$ of $H$ can be considered as a Frobenius 
complement of $FH$. Now we describe the action $H^f$ 
on the homogeneous components $L_i$.  Since $H$ is cyclic, we can choose 
a generator $h$ of $H$ and $r \in \{1,2,\ldots,p-1\}$ such that $\varphi^{h^{-1}} = \varphi^r$. 
Then $r$ is a primitive $q$th root of 1 in 
$\mathbb{Z}/p\mathbb{Z}$ (in this section $q$ denotes the order of $H$ and it is not necessarily a prime number). 
The group $H$ permutes the homogeneous components $L_i$ as follows: 
$L_i^h = L_{ri}$ for all $i \in \mathbb{Z}/p\mathbb{Z}$. Indeed, if 
$x_i \in L_i$, then $(x_i^h)^\varphi = x_i^{h\varphi h^{-1}h} = (x_i^{\varphi^r})^h = 
\omega^{ri}x_i^h$. On the other hand, since $F$ commutes with $\varphi$, we also 
have $L_i^{h^f} = L_{ri}$ for all 
$i \in \mathbb{Z}/p\mathbb{Z}$ (because if $x_i \in L_i$, then 
$(x_i^{h^f})^\varphi = \omega^{ri}x_i^{h^f}$). Thus, the action of $H^f$ 
on the components $L_i$ coincides with the action of $H$.

To lighten the notation we establish the following convention.

\begin{remark}[Index Convention]
In what follows, for a given $u_s \in L_s$ we denote both $u_s^{h^i}$ and 
$u_s^{(h^f)^i}$ by $u_{r^is}$, 
since $L_s^{h^i} = L_s^{(h^f)^i}=L_{r^is}$. Therefore, we can write 
$u_s + u_{rs} + \cdots + u_{r^{q-1}s}$ to mean a fixed-point of $H^f$ for any $f\in F$.
\end{remark}

\noindent\textbf{Proof of Theorem \ref{Liealgthe}} \\

We use the arguments presented at the beginning of this section.

\begin{lemma}\label{L_0}
There exists a $(c,d,|FH|)$-bounded number $u$ such that $[L,\underbrace{L_0,\ldots,L_0}_u]=0$.
\end{lemma}
\begin{proof} It suffices to prove that $[L_b,\underbrace{L_0,\ldots,L_0}_u]=0$ for any
$b \in \{0,1,\ldots, p-1\}$.

Taking into account that $L_0$ is nilpotent of class $d$, we may assume $b \neq 0$.

By Lemma \ref{lemma dec L0}, $L_0 =  C_L(F) + \sum_f V_f$.
We now by Lemma \ref{lemma dec L}, that
$[L,\underbrace{C_L(F),\ldots,C_L(F)}_d]=0$. Now let $u_b \in L_b$. 
Since $u_b + u_{rb} + \cdots + u_{r^{q-1}b} \in C_L(H^f)$ and 
$V_f \subseteq C_L(H^f)$, we have
$$[u_b + u_{rb} + \cdots + u_{r^{q-1}b}, \underbrace{V_f,\ldots,V_f}_c] = 0.$$
Thus, $\sum_{i=0}^{q-1} [u_{r^ib}, \underbrace{V_f,\ldots,V_f}_c] = 0$. 
On the other hand, $[u_{r^ib}, \underbrace{V_f,\ldots,V_f}_c] \in L_{r^ib}$ 
and $L_{r^ib}\neq L_{r^jb}$ whenever $i \neq j$. Therefore, we obtain $[L_b, \underbrace{V_f,\ldots,V_f}_c] = 0$.

Let $S$ be an $FH$-invariant subalgebra of $L_0$. Now let $S_F = S \cap C_L(F)$ and 
let $S_f = S \cap V_f$ for $f \in F$. 
It follows that $S = S_F + \sum_f S_f$. We will show that there exists a $(c,d,|FH|)$-bounded 
number $u$ such that $[L_b,\underbrace{S,\ldots,S}_u]=0$ using induction on the 
nilpotency class of $S$. Since $[S,S]$ is nilpotent of smaller class, 
there exists a $(c,d,|FH|)$-bounded number $u_1$ such that $[L_b,\underbrace{[S,S],\ldots,[S,S]}_{u_1}]=0$. 

Set $l=d+(c-1)\left|F\right|$ and 
$W=[L_b,S_{i_1},\ldots,S_{i_l}]$ for some choice of $S_{i_j}$ in $S_F \cup \{S_{f} ; \ f\in F \}$. 
It is clear that for any permutation $\pi$ of the symbols $i_1,\ldots,i_l$ we have 
$W\leq [L_b,S_{\pi(i_1)},\ldots,S_{\pi(i_l)}]+[L_b,[S,S]]$. Note that the number $l$ 
is big enough to ensure that $S_F$ occurs at least $d$ times or some $S_j \in \{S_{f} ; \ f\in F \}$ 
occurs at least $c$ times in the list $S_{i_1},\ldots,S_{i_l}$. 
Thus, if $m = \max\{c,d\}$ we deduce that there exists $i$ such that 
$W\leq [[L_b,\underbrace{S_i,\ldots,S_i}_{m}], *,\ldots,*]+[L_b,[S,S]]$,  
where the asterisks denote some of the subalgebras $S_{i_j} \in S_F \cup \{S_{f} ; \ f\in F \}$ which, in view of the fact that 
$[L_b,\underbrace{S_i,\ldots,S_i}_{m}]=0$, are of no consequence. Hence, $W\leq [L_b,[S,S]]$.

Further, for any choice of $S_{i_1},\ldots,S_{i_l} \in S_F \cup \{S_{f} ; \ f \in F \}$ the same 
argument shows that $$[W,S_{i_1},\ldots,S_{i_l}]\leq [W,[S,S]]\leq [L_b,[S,S],[S,S]].$$ 
More generally, for any $n$ and any $S_{i_1},\ldots,S_{i_{nl}} \in S_F \cup \{S_f; \ f \in F\}$ 
we have $$[L_b,S_{i_1},\ldots,S_{i_{nl}}]\leq [L_b,\underbrace{[S,S],\dots,[S,S]}_n].$$

Put $u=u_1l$. The above shows that 
$$[L_b,S_{i_1},\ldots,S_{i_{u}}]\leq [L_b,\underbrace{[S,S],\dots,[S,S]}_{u_1}]=0.$$ 
Of course, this implies that $[L_b,\underbrace{S,\ldots,S}_u]=0$. 
The lemma is now straightforward from the case where $S=L_0$.
\end{proof}

In view of Lemma \ref{L_0}, Theorem \ref{Liealgthe} holds if we show the following:

\begin{lemma}\label{metablemma} Let $L$ be metabelian. There exists a 
$(c,d,\left|FH\right|)$-bounded number $v$ such that $[[L,L] \cap L_0,\underbrace{L_a,\ldots,L_a}_{v}]=0, 
\ \mbox{for all} \ a \in \mathbb{Z}/p\mathbb{Z}$.
\end{lemma}
\begin{proof}
Recall that $L_0 =  C_L(F) + \sum_f V_f$ and set $L'_F = [L,L] \cap C_L(F)$. 
We have $[L,L] \cap L_0 = L'_F  + \sum_f [L,L]\cap V_f$ where $V_f=C_L(ZH^f)$. Set 
$V=C_L(ZH)$ and $V'=[L,L]\cap V$. 

First we prove that 
$[L'_F,\underbrace{L_a,\ldots,L_a}_{v}]=0$ for any 
$a \in \mathbb{Z}/p\mathbb{Z}$.

Note that we can write $L_a = \sum_{y \in F \setminus Z} C_{L_a}(y)$ and let $k$ be the number of 
elements of $F \setminus Z$. Therefore,

$$[L'_F,\underbrace{L_a,\ldots,L_a}_{k(d-1)+1}] = 
\sum_{y \in F \setminus Z} [L'_F,\underbrace{C_{L_a}(y),\ldots,C_{L_a}(y)}_{d},*,\ldots,*] = 0.$$

Here the asterisks denote some subspaces of $L_a$ which are not important since the
commutator is 0 anyway.

Second we prove that $[V',\underbrace{L_a,\ldots,L_a}_{v}]=0$ for any 
$a \in \mathbb{Z}/p\mathbb{Z}$.

For any $u_a \in L_a$ we have $u_a + u_{ra} + \cdots + u_{r^{q-1}a} \in C_L(H)$ (under Index Convention). Thus, 
$$[V',\underbrace{u_a + u_{ra} + \cdots + u_{r^{q-1}a},\ldots,v_a + v_{ra} + \cdots + v_{r^{q-1}a}}_c] = 0$$
for any $c$ elements $u_a,\ldots,v_a \in L_a$.

Let $T$ denote the span of all the sums $x_a + x_{ra} + \cdots + x_{r^{q-1}a}$ over $x_a \in L_a$ 
(in fact, $T$ is the fixed-point subspace of $H$ on $\bigoplus_{i=0}^{q-1}L_{r^ia}$). Then the 
latter equality means that
$$[V',\underbrace{T,\ldots,T}_c] = 0.$$

Applying $\varphi^j$ we also obtain
$$[V',\underbrace{T^{\varphi^j},\ldots,T^{\varphi^j}}_c] = 
[(V')^{\varphi^j},\underbrace{T^{\varphi^j},\ldots,T^{\varphi^j}}_c] = 0.$$

A Vandermonde-type linear algebra argument shows that $L_a \subseteq \sum_{j=0}^{q-1}T^{\varphi^j}$. 
Actually this fact is a consequence of the following result:

\begin{lemma}\cite[Lemma 5.3]{KhuMakShu}\label{auxlemma} Let $\left\langle \alpha\right\rangle$ 
be a cyclic group of order $n$, and $\omega$ a primitive nth root of unity. 
Suppose that $M$ is a $\mathbb{Z}[\omega]\left\langle\alpha\right\rangle$-module 
such that $M = \sum_{i=1}^{m} M_{t_i}$,
where $x\alpha = \omega^{t_i}x$ for $x \in M_{t_i}$ and 
$0 \leq t_1 < t_2 < \cdots < t_m < n$. 
If $z = y_{t_1} + y_{t_2} + \cdots + y_{t_m}$ for $y_{t_i} \in M_{t_i}$, 
then for some $m$-bounded number $d_0$ every element $n^{d_0}y_{t_s}$ is 
a $\mathbb{Z}[\omega]$-linear combination of the elements 
$z, z\alpha, \ldots, z\alpha^{m-1}$.
\end{lemma}

Now we can apply Lemma \ref{auxlemma} with $\alpha = \varphi$, 
$M = L_a + L_{ra} + \cdots + L_{r^{q-1}a}$ 
and $m = q$ to $w = u_a + u_{ra} + \cdots + u_{r^{q-1}a} \in T$ for any $u_a \in L_a$, 
because here the indices $r^ia$ can be regarded as pairwise distinct residues modulo 
$p$ ($r$ is a primitive $q$th root of 1 modulo $p$). Since $p$ is invertible in our 
ground ring, follows that $L_a \subseteq \sum_{j=0}^{q-1}T^{\varphi^j}$.

Set $v = (c-1)q + 1$. We now claim that $[V',\underbrace{L_a,\ldots,L_a}_v]=0$. Indeed, after replacing 
$L_a$ with $\sum_{j=0}^{q-1}T^{\varphi^j}$ and expanding the sums, in each commutator 
of the resulting linear combination we can freely permute the entries $T^{\varphi^j}$, since 
$L$ is metabelian. Since there are sufficiently many of them, we can place at least $c$ 
of the same $T^{\varphi^{j_0}}$ for some $j_0$ right after $V'$ at the beginning, which gives 0.

Note that in the case where $f \in F\setminus Z$ we can consider the Frobenius 
group $ZH^f$ and conclude in a similar way that $$[[L,L]\cap V_f,\underbrace{u_a,\ldots,v_a}_v] = 0$$
for any $v$ elements $u_a,\ldots,v_a \in L_a$.

\end{proof}

\section{Proof of Theorem \ref{main2}}

By Lemma \ref{submet} we can assume that $F$ is an elementary abelian group of order $p^2$ for some prime $p$. We know that $G$ is nilpotent and there exists a constant $P=P(c,|H|)$ such that if $p> P$, then the nilpotency class of $G$ is $(c,|H|)$-bounded  (see \cite[Theorem 1.5]{AS}). We wish to show that the nilpotency class of $G$ is $(c,|H|)$-bounded when $p\leq P$. In particular, we can assume that $|F|$ is $(c,|H|)$-bounded.
 
It is easy to see that without loss of generality we may assume that $G$ is a $p$-group. Consider the associated Lie ring of the group 
$G$ $$L(G)=\bigoplus_{i=1}^{n}\gamma_i/\gamma_{i+1},$$ where $n$ is the nilpotency class of $G$ and $\gamma_i$ are the terms of the lower central series of $G$. The nilpotency class of $G$ coincides with the nilpotency class of $L(G)$. The action of the group $FH$ on $G$ induces naturally an action of $FH$ on $L(G)$ and the subrings $C_{L(G)}(H)$ and $C_{L(G)}(x)$ are nilpotent of class at most $c$ for any $x\in F^{\#}$ since the action is coprime. Theorem \ref{Liealgthe} now tells us that $L(G)$ is nilpotent 
of $(c,|H|)$-bounded class. The proof is complete.

\end{document}